\documentclass[11pt, reqno]{tdb}

\RequirePackage{mathrsfs} \let\mathcal\mathscr

\usepackage{amsfonts, amsthm, amsmath, amssymb,a4}

\numberwithin{equation}{section}

\newcommand{\QQ}{\mathbb{Q}}
\newcommand{\PP}{\mathbb{P}}

\newcommand{\RR}{\mathbb{R}}
\newcommand{\CC}{\mathbb{C}}
\newcommand{\NN}{\mathbb{N}}
\newcommand{\ZZ}{\mathbb{Z}}

\newcommand{\sfl}{\mathsf{\Lambda}}

\newcommand{\bnu}{\boldsymbol{\nu}}

\newcommand{\bal}{\boldsymbol{\alpha}}
\newcommand{\bbe}{\boldsymbol{\beta}}
\newcommand{\bep}{\boldsymbol{\epsilon}}

\newcommand{\bxi}{\boldsymbol{\xi}}
\newcommand{\bta}{\boldsymbol{\tau}}

\newtheorem{thm}{Theorem}
\newtheorem{lem}{Lemma}

\newcommand{\bd}{\mathbf{d}}
\newcommand{\y}{\mathbf{y}}
\newcommand{\z}{\mathbf{z}}

\renewcommand{\v}{\mathbf{v}}
\renewcommand{\rho}{\varrho}
\renewcommand{\leq}{\leqslant}
\renewcommand{\le}{\leqslant}
\renewcommand{\geq}{\geqslant}

\newcommand{\x}{{\bf x}}

\newcommand{\ma}{\mathbf}
\newcommand{\ve}{\varepsilon}

\newcommand{\mcal}{\mathcal}
\newcommand{\lab}{\label}
\newcommand{\al}{\alpha}

\newcommand{\D}{\Delta}
\newcommand{\be}{\beta}
\newcommand{\la}{\lambda}

\renewcommand{\d}{\mathrm{d}}
\renewcommand{\leq}{\leqslant}
\renewcommand{\geq}{\geqslant}

\theoremstyle{definition}

\DeclareMathOperator{\vol}{vol}

\DeclareMathOperator{\Mod}{mod}
\renewcommand{\bmod}[1]{\,(\Mod{ #1})}

\newcommand{\mnu}{\boldsymbol{\nu}}

\begin{document}

\title[]{The divisor problem for binary cubic forms}

\author[]{T.D. Browning}

\address{School of Mathematics\\
University of Bristol\\ Bristol BS8 1TW\\
United Kingdom}

\email{t.d.browning@bristol.ac.uk}

\subjclass{11N37 (11D25)}

\date{\today}

\begin{abstract}
We investigate the average order of the divisor function at values of
binary cubic forms that are reducible over $\QQ$ and discuss some
applications.
\end{abstract}

\maketitle

\tableofcontents

\section{Introduction}

This paper is motivated by the well-known problem of studying the
average order of the divisor function 
$\tau(n)=\sum_{d\mid n}1$, as it ranges over the values taken by
polynomials. 
Our focus is upon the case of binary forms 
$C\in \ZZ[x_1,x_2]$ of degree $3$, the treatment of degree
$1$ or $2$ being essentially trivial.

We wish to understand the behaviour of the sum
$$
  T(X;C)=\sum_{x_1,x_2\leq X}\tau(C(x_1,x_2)),
$$
as $X\rightarrow \infty$. The hardest case is when 
$C$ is irreducible over $\QQ$ with  non-zero discriminant, a situation
first handled by Greaves \cite{greaves}. He establishes the existence of 
constants $c_0,c_1\in\RR$, with $c_0>0$, such that
$$
T(X;C)= c_0X^2\log X+c_1X^2 +O_{\ve, C}(X^{2-\frac{1}{14}+\ve}),
$$
for any $\ve>0$.  
Here, as throughout our work, any dependence in the implied constant
will be indicated explicitly by an appropriate subscript. 
This was later improved  by Daniel \cite{D99}, who sharpened the
exponent $2-\frac{1}{14}+\ve$ to $2-\frac{1}{8}+\ve$.
Daniel also achieves asymptotic information about the 
sum associated to irreducible binary forms of degree $4$, which 
is at the limit of what is currently possible.

Our aim 
is to investigate the corresponding sums $T(X)=T(X;L_1L_2L_3)$
when $C$ is assumed to factorise as a product of  
linearly independent linear forms $L_1,L_2,L_3\in \ZZ[x_1,x_2]$.
In doing so we will gain a respectable improvement in the quality
of the error term apparent in the work of Greaves and Daniel. 
The following result will be established in \S
\ref{s:4}.

\begin{thm}\lab{t:d}
For any $\varepsilon>0$ there exist constants $c_0,\ldots,c_3\in
\RR$, with $c_0>0$, such that 
$$
T(X)
= \sum_{i=0}^3c_iX^2(\log X)^{3-i}+
O_{\ve, L_1,L_2,L_3}( X^{2-\frac{1}{4}+\ve}).
$$
\end{thm}

Our proof  draws heavily on a series of joint papers of
the author with la
Bret\`eche \cite{4linear, L1L2Q}. These involve 
an analysis of the more
exacting situation wherein $\tau(L_1L_2L_3)$ is replaced by 
$r(L_1L_2L_3L_4)$ or $\tau(L_1L_2Q)$, for an irreducible binary quadratic 
form $Q$.

One of the motivations for studying the divisor problem for binary
forms is the relative lack of progress  for the
divisor problem associated to polynomials in a single variable. 
It follows from work of Ingham \cite{ing} that
$$
\sum_{n\leq X} \tau(n)\tau(n+h)\sim 
\frac{6}{\pi^2}\sigma_{-1}(h)X(\log X)^2
$$
as $X
\rightarrow \infty$, for given $h \in \NN$.
Exploiting connections with Kloosterman sums, Estermann  \cite{est}
obtained a cleaner asymptotic expansion with
a reasonable degree of uniformity in $h$. Several authors have since 
revisited this problem achieving asymptotic formulae 
with $h$ in an increasingly
large range compared to $X$. The best results in the literature are
due to Duke, Friedlander and Iwaniec \cite{dfi} and to 
Motohashi \cite{moto}.

A successful analysis of the sum 
$$
T_h(X)=\sum_{n\leq X}\tau(n-h)\tau(n)\tau(n+h),
$$
has not yet been forthcoming for a single positive integer $h$. It is
conjectured that $T_h(X)\sim c_h X(\log X)^3$
as $X \rightarrow \infty$, for a suitable constant $c_h>0$. A
straightforward heuristic analysis based on the underlying Diophantine
equations suggests that one should take
\begin{equation}
  \label{eq:ch}
  c_h=\frac{11}{8}f(h)\prod_{p}\Big(1-\frac{1}{p}\Big)^2\Big(1+\frac{2}{p}\Big),
\end{equation}
where $f$ is given multiplicatively by $f(1)=1$ and 
\begin{equation} \label{eq:cf}
  f(p^\nu)=\begin{cases}
(1+\frac{2}{p})^{-1}(1-\frac{1}{p})^{-2}(1+\frac{4}{p}+\frac{1}{p^2}-\frac{3\nu+4}{p^{\nu+1}}
-\frac{4}{p^{\nu+2}}+\frac{3\nu+2}{p^{\nu+3}}), & \mbox{if
  $p>2$,}\\
\frac{52}{11}-\frac{41+15\nu}{11\times 2^{\nu}}, 
& \mbox{if
  $p=2$,}
\end{cases}
\end{equation}
for $\nu \geq 1$.
In the following result we
provide some support for this expectation. 

\begin{thm}\lab{t:unary}
Let $\varepsilon>0$ and let 
$H\geq X^{\frac{3}{4}+\ve}$. Then we have  
$$
\sum_{h\leq H}\big( T_h(X)-c_hX(\log X)^3\big) =o(HX(\log X)^3).
$$
\end{thm}

This result will be established in \S \ref{s:5}, where 
we will see that $HX(\log X)^3$ represents the true order of magnitude
of the two sums  on the left hand side.
It would be interesting to reduce the lower bound for $H$ assumed in 
this result. 

Throughout our work it will be convenient to reserve $i,j$ for 
generic distinct indices from the set $\{1,2,3\}$. 
For any $\ma{h}\in\NN^3$, we let 
\begin{align}\label{eq:farm} 
\mathsf{\Lambda}({\ma{h}})&=
 \{ \x \in \ZZ^2 : h_i \mid L_i(\x)\}, \\
\rho(\ma{h})&=
  \# \big(\mathsf{\Lambda}({\ma{h}})\cap  [0,h_1h_2h_3)^2 \big).
  \label{defrho}
\end{align}
It is clear that $\mathsf{\Lambda}(\ma{h})$ defines an integer
sublattice of rank $2$.  
In what follows let $\mcal{R}$ always denote a compact subset of 
$\RR^2$ 
whose boundary is a
piecewise continuously
differentiable closed curve with length
$$  \partial(\mcal{R})\leq 
\sup_{\x\in\mcal{R}}\max\{|x_1|,|x_2|\}.
$$
This is in contrast to our earlier investigations \cite{4linear,
  L1L2Q}, where a hypothesis of this sort is automatically satisfied
by working with closed convex subsets of $\RR^2$.
Let $\ma{d}, \ma{D}\in \NN^3$ such
that $d_i\mid D_i$.  
We shall procure Theorems~\ref{t:d} and
\ref{t:unary} through an
analysis of the auxiliary sum
\begin{equation}
   \label{eq:Sj}
S(X;\ma{d},\ma{D})=\sum_{\substack{\x\in \sfl(\ma{D})\cap
     X\mcal{R}}}
\tau\Big(\frac{L_1(\x)}{d_1}\Big)
\tau\Big(\frac{L_2(\x)}{d_2}\Big)\tau\Big(\frac{L_3(\x)}{d_3}\Big),
\end{equation}
where $X\mcal{R}= \{X\x: \x\in \mcal{R}\}.$  We will also assume that
$L_i(\x)>0$ for $\x \in \mcal{R}$. 

Before revealing our estimate for $S(X;\ma{d},\ma{D})$ we will first need to introduce some
more notation. 
We write
\begin{equation}
  \label{eq:Linf}
L_\infty=L_\infty(L_1,L_2,L_3)=\max \{\|L_1\|,
\|L_2\|,\|L_3\|\},
\end{equation}
where $\|L_i\|$ denotes the maximum modulus of the coefficients of
$L_i$. We will set 
\begin{align}
  \label{eq:rinf}
r_\infty&=r_\infty(\mcal{R})=\sup_{\x\in\mcal{R}}\max\{|x_1|,|x_2|\},\\
  \label{eq:ri'}
r'&=r'(L_1,L_2,L_3,\mcal{R})
=\max_{1\leq i\leq 3}\sup_{\x\in\mcal{R}} \{L_i(\ma{x})\}.
\end{align}
These are positive real numbers by assumption. 
Furthermore, let $D=D_1D_2D_3$ and let $\delta(\ma{D})\in \NN$ denote the largest 
$\delta\in \NN$ for which $\sfl(\ma{D})\subseteq \{\x \in \ZZ^2:
\delta\mid \x\}$. Bearing this notation in mind we will establish the following result
in \S \ref{s:2} and \S \ref{s:3}.

\begin{thm}\lab{main1}
Let $\varepsilon>0$ and let $\theta \in (\frac{1}{4},1)$. 
Assume that $r' X^{1-\theta}\geq 1$.
Then there exists a polynomial $P\in \RR[x]$ of degree $3$ such
that
\begin{align*}
S(X;\ma{d},\ma{D})=~& \vol(\mcal{R})X^2P(\log X)  \\
&+
O_\ve\Big( 
\frac{D^\ve  L_\infty^{2+\ve} r_\infty^\ve }{\delta(\ma{D})} 
\big(r_\infty r'^{\frac{3}{4}}+ 
r_\infty^2  \big)X^{\frac{7}{4}+\ve}
\Big),
\end{align*}
where the coefficients of $P$ have modulus 
$O_\ve (D^\ve L_\infty^\ve r_\infty^\ve(1+r'^{-1})^\ve 
(\det \sfl(\ma{D}))^{-1}).$
Moreover, the leading coefficient of $P$ 
is $C=\prod_p \sigma_p(\ma{d},\ma{D})$, with  
\begin{equation}
    \label{defsigma'}
\sigma_p(\ma{d},\ma{D}) =\Big(1-\frac{1}{p}\Big)^3
\sum_{\mnu\in\ZZ_{\geq 0}^3} \frac{\rho
(p^{N_1},p^{N_2},p^{N_3})}{p^{2N_1+2N_2+2N_3}}
\end{equation}
and $
N_i=\max\{v_p(D_i),\nu_i+v_p(d_i)\}.
$
\end{thm}


While the study of the above sums is
interesting in its own right, it turns out that there are useful 
connections to conjectures of Manin and his collaborators \cite{bm} 
concerning
the growth rate of rational points on Fano varieties.  
Consider for example the bilinear hypersurface
$$
W_s:\quad x_0y_0+\cdots +x_sy_s=0
$$
in $\PP^s\times \PP^s.$   This defines a flag variety and 
it can be embedded in
$\PP^{s(s+2)}$ via the Segre embedding $\phi$. 
Let $U_s\subset W_s$ be the open subset on which $x_iy_j\neq 0$ for
$0\leq i,j\leq n$. 
If $H:
\PP^{s(s+2)}(\QQ)\rightarrow \RR_{>0}$ is the usual exponential height 
then we wish to analyse the counting function
\begin{align*}
N(B)
&=\#\{v\in U_s(\QQ): H(\phi(v))\leq B\}\\
&=\frac{1}{4}\#\{
(\x,\y)\in \ZZ_{*}^{s+1}\times \ZZ_{*}^{s+1}: 
~\max |x_iy_j|^s\leq B, ~\x.\y=0\},
\end{align*}
as $B\rightarrow \infty$,
where $\ZZ_*^k$ denotes the set of primitive vectors in $\ZZ^k$ with non-zero
components.
 It follows from work of Robbiani
\cite{robby} that there is a constant $c_s>0$ such that 
$N(B)\sim c_sB\log B$, for $s\geq 3$, which thereby confirms the Manin
conjecture in this case.  This is established using the Hardy--Littlewood circle
method. Spencer \cite{spencer} has given a substantially
shorter treatment, which also handles the case $s=2$. 
By casting the problem in terms of a restricted divisor sum in 
\S  \ref{s:bil}, we will modify the proof of 
Theorem \ref{main1} to provide an independent proof of Spencer's
result in the case $s=2$.

\begin{thm}\label{thm:bil}
For $s=2$ we have 
$N(B)=cB \log B+O(B)$, with 
$$
c=\frac{12}{\zeta(2)^2}
\prod_p
\Big(1+\frac{1}{p}\Big)^{-1}\Big(1+\frac{1}{p}+\frac{1}{p^2}\Big).
$$ 
\end{thm}

\section{Theorem \ref{main1}: special case}\label{s:2}

Our proof follows the well-trodden paths of \cite[\S
 4]{4linear} and \cite[\S\S 5,6]{L1L2Q}. 
We will begin by establishing a version of Theorem \ref{main1} when
 $d_i=D_i=1$. Let us write $S(X)$ for the sum in this special
 case. In \S \ref{s:3} we shall establish the general case by reducing the
 situation to this  case via a linear change of variables. 

Recall that the linear forms under consideration are not necessarily primitive. 
We therefore fix integers $\ell_i$ such that $L_i^*$ is a primitive linear
form, with 
\begin{equation}\label{ecritprim}
L_i=\ell_iL_i^*.
\end{equation}
It will be convenient to define the least common multiple
\begin{equation}\label{eq:Lg}
L_*=[\ell_1,\ell_2,\ell_3].
\end{equation}
Let $\varepsilon>0$ and assume that $r' X^{1-\psi}\geq 1$
for some parameter $\psi \in (0,1)$. 
Throughout our work 
we will follow common practice and allow the small parameter
$\ve>0$ to take different values at different parts of the argument,
so that $x^\ve \log x\ll_\ve x^\ve$, for example. 
In this section we will show that 
there exists a polynomial $P\in \RR[x]$ of degree $3$ such
that 
\begin{equation}
  \label{eq:24}
\begin{split}
S(X)
=~& \vol(\mcal{R})X^2P(\log X)\\
&+
O_\ve\big(L_\infty^\ve r_\infty^\ve r'^{\frac{3}{4}}
(r_\infty+L_*^{\frac{1}{2}}\vol(\mcal{R})^{\frac{1}{2}}
\big)X^{\frac{7}{4}+\ve}, 
\end{split}
\end{equation}
where the leading coefficient of $P$ is $\prod_p \sigma_p$, with  
\begin{equation}
    \label{defsigma}
\sigma_p =\Big(1-\frac{1}{p}\Big)^3
\sum_{\mnu\in\ZZ_{\geq 0}^3} \frac{\rho
(p^{\nu_1},p^{\nu_2},p^{\nu_3})}{p^{2\nu_1+2\nu_2+2\nu_3}}.
\end{equation}
Moreover, 
the coefficients of $P$ have modulus $O_\ve(L_\infty^{\ve}
r_\infty^\ve(1+r'^{-1})^\ve)$. 

As a first step we deduce from the trivial bound for the divisor
function the estimate
\begin{equation}
    \label{majSXtriv}
S(X) \ll_\ve  L_\infty^\varepsilon r_\infty^{2+\ve}X^{2+\ve}.
\end{equation}
We will also need to record the inequalities 
\begin{equation}
    \label{prelim}
\frac{r'}{2L_\infty}\leq r_\infty\leq 2r' L_\infty,\quad \vol(\mcal{R})\leq 
4r_\infty^2.
\end{equation}
The lower bounds for $r_\infty$ and $4r_\infty^2$ are trivial. To see
the remaining bound 
we suppose that $L_i(\x)=a_ix_1+b_ix_2$. 
Let $\Delta_{i,j}=a_ib_j-a_jb_i$ denote the resultant of $L_i,L_j$. By
hypothesis $\Delta_{i,j}$ is a non-zero integer. We have 
$$
x_1=\frac{b_j L_i(\x)-b_iL_j(\x)}{\Delta_{i,j}}, \quad
x_2=\frac{a_iL_j(\x)-a_j L_i(\x)}{\Delta_{i,j}}, 
$$
for any $i,j$. It therefore follows that 
$r_\infty\leq 2 r' L_\infty$, as required for \eqref{prelim}.

The technical tool underpinning the proof of \eqref{eq:24} is an
appropriate ``level of distribution'' result. 
Recall the definitions \eqref{eq:farm} and \eqref{defrho}. 
The following  is a trivial modification of the proofs of 
\cite[Lemma 3]{4linear} and \cite[Lemma 3.2]{D99}.

\begin{lem}\label{LOD}
Let $\varepsilon>0$. Let $X\geq 1$, $Q_i \geq 2$ and 
$Q =Q_1 Q_2 Q_3$. Then there exists an absolute constant $A>0$
such that
\begin{align*} 
  \sum_{\substack{\ma{d}\in\NN^3\\ d_i\le Q_i}}
\Big|\#\big(\mathsf{\Lambda}({\ma{d}})\cap X\mcal{R}_{\ma{d}}\big)
  &-  \frac{\vol(X\mcal{R}_{\ma{d}}) \rho(\bd) }{(d_1d_2d_3)^2}
\Big| \\
& \ll_\ve  L_\infty^{\varepsilon}(M X(\sqrt{Q}+\max Q_i)+Q)(\log Q)^{A} ,
\end{align*}
where $\mcal{R}_{\ma{d}}\subseteq \mcal{R}$ is any 
compact set depending on $\ma{d}$
whose boundary is a
piecewise continuously
differentiable closed curve of length at most 
$M$.
\end{lem}

Recall the definition of $r'$ from 
\eqref{eq:ri'}.
In what follows it will be convenient to set
$$
X'=r' X.
$$
For any $1\leq i \leq 3$ and $\x \in X\mcal{R}$ we have 
\begin{equation}\label{eq:bone}
\begin{split}
\tau(L_i(\x))
&=
\sum_{\substack{d\mid L_i(\x)\\ d\leq \sqrt{X'}}}1+
\sum_{\substack{d\mid L_i(\x)\\ d> \sqrt{X'}}}1\\
&=
\sum_{\substack{d\mid L_i(\x)\\ d\leq \sqrt{X'}}}1+
\sum_{\substack{e\mid L_i(\x)\\ e\sqrt{X'}< L_i(\x)}}1\\
&=\tau_+(L_i(\x))+\tau_-(L_i(\x)),
\end{split}
\end{equation}
say. In this way we may produce a decomposition into $8$ subsums
\begin{equation}\label{eq:star}
S(X)=\sum S_{\pm,\pm,\pm}(X),
\end{equation}
where
$$
S_{\pm,\pm,\pm}(X)=\sum_{\x \in \ZZ^2\cap X\mcal{R}} \tau_\pm(L_1(\x))
\tau_\pm(L_2(\x))\tau_\pm(L_3(\x)).
$$
Each sum $S_{\pm,\pm,\pm}(X)$ is handled in the same way. Let us treat
the sum $S_{+,+,-}(X)$, which is typical.

On noting that $L_i(\x)\leq X'$ for any $\x \in X\mcal{R}$ we deduce
that 
$$
S_{+,+,-}(X)= \sum_{d_1,d_2,d_3\leq \sqrt{X'}}\#(\sfl(\ma{d})\cap
\mcal{S}_{\ma{d}}),
$$
where 
$\mcal{S}_{\ma{d}}$ is the set of $\x\in X\mcal{R}$ for which
$d_3\sqrt{X'}< L_3(\x)$.
To estimate this sum we apply Lemma
\ref{LOD} with $Q_1=Q_2=Q_3=\sqrt{X'}$. This gives
\begin{align*}
S_{+,+,-}(X)-&
\sum_{d_1,d_2,d_3\leq \sqrt{X'}} 
  \frac{\rho(\bd) \vol (\mcal{S}_{\ma{d}})}{(d_1d_2d_3)^2}\\
&\ll_\ve  L_\infty^\ve r_\infty^\ve \big(r_\infty
r'^{\frac{3}{4}} X^{\frac{7}{4}+\ve}
+r'^{\frac{3}{2}}X^{\frac{3}{2}+\ve}\big), 
\end{align*}
since $\partial(\mcal{R})\leq r_\infty$.
If $r'^{\frac{3}{4}}\leq r_\infty X^{\frac{1}{4}}$ then this error term
is satisfactory for \eqref{eq:24}. Alternatively, if
$r'^{\frac{3}{4}}> r_\infty X^{\frac{1}{4}}$, then the conclusion
follows from  \eqref{majSXtriv} instead. 
It remains to analyse the main term, the starting point for which is
an analysis of the sum
\begin{equation}
  \label{eq:MM}
M(\mathbf{T})=\sum_{d_i\leq T_i} 
  \frac{\rho(\bd) }{(d_1d_2d_3)^2},
\end{equation}
for $T_1,T_2,T_3\geq 1$. We will establish the following result.

\begin{lem}\label{lem:Masy}
Let $\ve>0$ and $T=T_1T_2T_3$.
Then there exist $c,c_{i,j},c_k,c_0 \in \RR$,
with modulus $O_\ve(L_\infty^\ve)$, such that 
\begin{align*}
  M(\mathbf{T})=~&
c\prod_{i=1}^3\log T_i +
\sum_{1\leq i<j\leq 3}c_{i,j} (\log T_i)(\log T_j)
+\sum_{1\leq k\leq 3}c_{k} \log T_k +c_0\\
&+O_\ve(L_\infty^\ve L_*^{\frac{1}{2}} \min\{T_1,T_2,T_3\}^{-\frac{1}{2}}T^\ve),
\end{align*}
where $L_*$ is given by \eqref{eq:Lg} and 
\begin{equation}\label{eq:moon}
c=\prod_p \sigma_p.
\end{equation}
\end{lem}

Before proving this result we first show how it leads to 
 \eqref{eq:24}. For ease of notation we write $d_3=z$
and $f(z)=\vol(\mcal{S}_{\ma{d}})$. Let $Q=\sqrt{X'}.$ 
Since $f(Q)=0$, it follows from partial
summation that
$$
\sum_{d_1,d_2,d_3\leq Q} 
  \frac{\rho(\bd) \vol (\mcal{S}_{\ma{d}})}{(d_1d_2d_3)^2}
=-\int_{1}^Q f' (z) M(Q,Q,z) \d z,
$$
in the notation of \eqref{eq:MM}.
An application of Lemma \ref{lem:Masy} reveals that there exist constants
$c,a_1,\ldots,a_5\ll_\ve L_\infty^\ve$ such that 
\begin{align*}
  M(Q,Q,z)=~&
c(\log Q)^2(\log z)+a_1(\log Q)^2+a_2(\log Q)(\log z)\\
&+a_3\log Q +a_4
\log z +a_5 +O_\ve(L_\infty^\ve L_*^{\frac{1}{2}} z^{-\frac{1}{2}}Q^\ve),
\end{align*}
with $c$ given by \eqref{eq:moon}. However we claim that
$$
f'(z)\ll \vol(\mcal{R})^{\frac{1}{2}} QX.
$$ 
To see this we suppose that $L_3(\x)=a_3x_1+b_3x_2$ with $|a_3|\geq
|b_3|.$ Then 
\begin{align*}
-f'(z)
&=\lim_{\Delta\rightarrow 0} \Delta^{-1}\vol\{\x\in X\mcal{R}:
 zQ<L_3(\x)\leq (z+\Delta)Q\}\\
&=\lim_{\Delta\rightarrow 0} \Delta^{-1}\vol\left\{
\left(y_1,
\frac{y_2+zQ-a_3y_1}{b_3}\right)\in X\mcal{R}:  0<y_2\leq \Delta
Q\right\}\\
&\ll Q\vol(X\mcal{R})^{\frac{1}{2}},
\end{align*}
on making the change of variables $y_1=x_1$ and $y_2=L_3(\x)-zQ$.
This therefore establishes the claim and we see that the error term contributes
\begin{align*}
&\ll_\ve 
L_\infty^\ve L_*^{\frac{1}{2}}  Q^\ve \int_{1}^Q |f' (z)| 
z^{-\frac{1}{2}} \d z\\
&\ll_\ve L_\infty^\ve L_*^{\frac{1}{2}}   \vol(\mcal{R})^{\frac{1}{2}} 
Q^{\frac{3}{2}+\ve}X\\
&\ll_\ve L_\infty^\ve r_\infty^\ve L_*^{\frac{1}{2}}
 \vol(\mcal{R})^{\frac{1}{2}} r'^{\frac{3}{4}} X^{\frac{7}{4}+\ve}.
\end{align*}
Moreover, we have 
\begin{align*}
\int_{1}^Q f' (z) \d z = f(1)=
X^2 \vol(\mcal{R})+O(r_\infty QX ),
\end{align*}
and
\begin{align*}
\int_{1}^Q (\log z) f' (z) \d z 
&= -\int_1^Q \frac{f(z)}{z}\d z\\
&= -\int_{\x \in X\mcal{R}}\int_{1<z<Q^{-1}L_3(\x)} \frac{\d z \d \x}{z}\\
&= -\int_{\x \in X\mcal{R}}\big(\log L_3(\x)-\log Q\big) \d \x\\
&= -X^2\vol(\mcal{R}) \log Q+ b X^2,
\end{align*}
for a constant $b\ll_\ve L_\infty^\ve r_\infty^\ve \vol(\mcal{R}) (1+r'^{-1})^\ve$.
Putting everything together we conclude
\begin{align*}
\sum_{d_1,d_2,d_3\leq Q} 
  \frac{\rho(\bd) \vol (\mcal{S}_{\ma{d}})}{(d_1d_2d_3)^2}
=~&
2^{-3}\vol(\mcal{R})X^2P(\log X) \\
&+
O_\ve\big(L_\infty^\ve r_\infty r'^{\frac{1}{2}}X^{\frac{3}{2}+\ve}\big)\\
&+O_\ve\big( 
L_\infty^\ve r_\infty^\ve L_*^{\frac{1}{2}}
 \vol(\mcal{R})^{\frac{1}{2}} r'^{\frac{3}{4}} X^{\frac{7}{4}+\ve}
\big),
\end{align*}
for a suitable polynomial $P\in \RR[x]$ of degree $3$ with 
leading coefficient $\prod_p \sigma_p$ and 
all coefficients having modulus $O_\ve(L_\infty^{\ve}
r_\infty^{\ve}(1+r'^{-1})^\ve)$. 
The error terms in this expression are satisfactory for 
\eqref{eq:24}. 
Once taken in conjunction with the analogous estimates for the
remaining $7$ sums in \eqref{eq:star}, this therefore completes the proof of 
\eqref{eq:24}.

We may now turn to the proof of Lemma \ref{lem:Masy}, which 
rests upon an
explicit investigation of the function
$\rho(\ma{d})$.  Now it 
follows from the Chinese remainder theorem that there is a multiplicativity property
$$
\rho(g_1 h_1,g_2 h_2,g_3 h_3)=\rho(g_1,g_2,g_3)\rho(h_1,h_2,h_3),
$$
whenever $\gcd(g_1 g_2 g_3,h_1 h_2 h_3)=1$. 
Recall that $\Delta_{i,j}$ is used to denote the resultant of
$L_i,L_j$ and set
$$
\D=|\D_{1,2}\D_{1,3}\D_{2,3}|\neq 0.
$$
Recall the definition of $\ell_i$ and $L_i^*$ from \eqref{ecritprim}.
The following result collects together some information about the
behaviour of $\rho(\ma{d})$ at prime powers.

\begin{lem}\label{lem:rho}
Let $p$ be a prime.
Suppose that $\min\{e_i,\nu_p(\ell_i)\}=0$. Then we have 
$$
\rho(p^{e_1},1,1)=p^{e_1}, \quad 
\rho(1,p^{e_2},1)=p^{e_2}, \quad 
\rho(1,1,p^{e_3})=p^{e_3}.
$$
Next suppose that 
$0\leq e_i\leq e_j \leq e_k$ for a permutation $\{i,j,k\}$ of $\{1,2,3\}$.
Then we have
$$
\rho(p^{e_1},p^{e_2},p^{e_3})
\begin{cases}
=p^{2e_i+e_j+e_k}, &\mbox{if $p\nmid \D$,}\\
\leq p^{2e_i+e_j+e_k+\min\{e_j,v_p(\D)\}+\min\{e_k,v_p(\ell_k)\}},
&\mbox{if $p\mid \D$.}
\end{cases}
$$
\end{lem}

\begin{proof}
The first part of the lemma is obvious. To see the second part we 
suppose without loss of generality that $e_1\leq e_2\leq e_3$.  

When $p\nmid \D$ we see that 
the conditions $p^{e_i}\mid L_i(\x)$ in 
$\rho(p^{e_1},p^{e_2},p^{e_3})$ are equivalent to $p^{e_2}\mid
\x$ and $p^{e_3}\mid L_3(\x)$.
Thus we conclude that
$$
\rho(p^{e_1},p^{e_2},p^{e_3})=
\#\{\x \bmod{p^{e_1+e_3}}: p^{e_3-e_2}\mid L_3(\x)
\}=p^{2e_1+e_2+e_3},
$$
as required. 

Turning to the case  $p\mid \D$, 
we begin with the inequalities
\begin{align*}
\rho(p^{e_1},p^{e_2},p^{e_3})
&\leq p^{2e_1}
\rho(1,p^{e_2},p^{e_3})\\
&\leq p^{2e_1}
\#\{\x \bmod{p^{e_2+e_3}}: p^{e_2}\mid \D_{2,3}\x,~p^{e_3}\mid L_3(\x)
\}.
\end{align*}
Let us write $\delta=v_p(\D_{2,3})$ and $\la=v_p(\ell_3)$ for short. 
In particular it is clear that $\delta\geq \lambda$.
In this way we deduce that
$\rho(p^{e_1},p^{e_2},p^{e_3})$ is at most
\begin{align*}
p^{2e_1}
\#\{\x \bmod{p^{e_2+e_3}}: p^{\max\{e_2-\delta,0\}}\mid \x,~p^{\max\{e_3-\lambda,0\}}\mid L_3^*(\x)
\}.
\end{align*}
Suppose first that $e_2\geq \delta$. Then $0\leq e_2-\delta\leq
e_3-\lambda$ and it follows that 
\begin{align*}
\rho(p^{e_1},p^{e_2},p^{e_3})
&\leq p^{2e_1}
\#\{\x \bmod{p^{e_3+\delta}}: p^{e_3-\lambda}\mid
p^{e_2-\delta} L_3^*(\x)
\}\\
&= p^{2e_1}\cdot p^{e_3+\delta}\cdot p^{e_2+\lambda}\\
&= p^{2e_1+e_2+e_3+\delta+\lambda},
\end{align*}
since $L_3^*$ is primitive.  Alternatively, if 
$e_2<\delta$, we deduce that 
\begin{align*}
\rho(p^{e_1},p^{e_2},p^{e_3})
&\leq p^{2e_1}
\#\{\x \bmod{p^{e_2+e_3}}: p^{\max\{e_3-\lambda,0\}}\mid L_3^*(\x)
\}\\
&= p^{2e_1+2e_2+e_3+\min\{e_3,\lambda\}}.
\end{align*}
Taking together these two estimates 
completes the proof of the lemma. 
\end{proof}

We now have the tools in place with which to tackle the proof of Lemma
\ref{lem:Masy}. We will argue using Dirichlet
convolution, as in \cite[Lemma 4]{L1L2Q}. Let 
$$
f(\ma{d})=\frac{\rho(\ma{d})}{d_1d_2d_3}
$$
and let $h:\NN^3\rightarrow
\NN$ be chosen so that 
$
f(\ma{d})=(1*h)(\ma{d}),
$
where $1(\ma{d})=1$ for all $\ma{d}\in \NN^3$.  We then have 
$$
h(\ma{d})=(\mu*f)(\ma{d}),
$$
where $\mu(\ma{d})=\mu(d_1)\mu(d_2)\mu(d_3)$. The following result is
the key technical estimate in our analysis of $M(\mathbf{T})$.

\begin{lem}\label{lem:tail}
For any $\ve>0$ and any $\delta_1,\delta_2,\delta_3\geq 0$ such that
$\delta_1+\delta_2+\delta_3<1$, 
we have 
$$
\sum_{\ma{k}\in \NN^3}  \frac{|h(\ma{k})|}{k_1^{1-\delta_1}k_2^{1-\delta_2}k_3^{1-\delta_3}}
\ll_{\delta_1,\delta_2,\delta_3, \ve} L_\infty^\ve L_*^{\delta_1+\delta_2+\delta_3},
$$
where $L_*$ is given by \eqref{eq:Lg}.
\end{lem}

\begin{proof}
On noting that $k_1^{\delta_1}k_2^{\delta_2}k_3^{\delta_3}\leq
k_1^{\delta_\Sigma}+k_2^{\delta_\Sigma}+k_3^{\delta_\Sigma}$, with
$\delta_\Sigma=\delta_1+\delta_2+\delta_3$, it clearly suffices to
establish the lemma in the special case $\delta_2=\delta_3=0$ and
$0\leq \delta_1<1$.

Using the multiplicativity of $h$, our task is to estimate
the Euler product
$$
P=\prod_p \Big(
1+\sum_{\substack{\nu_i\geq 0\\ \bnu\neq \mathbf{0}}}
\frac{|h(p^{\nu_1},p^{\nu_2},p^{\nu_3})|p^{\nu_1\delta_1}}{p^{\nu_1+\nu_2+\nu_3}}
\Big)=\prod_p P_p,
$$
say. Now for any prime $p$, we deduce that
\begin{equation}
  \label{eq:weston}
  |h(p^{\nu_1},p^{\nu_2},p^{\nu_3})|=|(\mu*f)(p^{\nu_1},p^{\nu_2},p^{\nu_3})|\leq 
(1*f)(p^{\nu_1},p^{\nu_2},p^{\nu_3}),
\end{equation}
whence
$$
P_p\leq 1+\sum_{\substack{\al_i,\be_i\geq 0\\
    \bal+\bbe\neq \ma{0}}} \frac{p^{\al_1\delta_1}}{p^{\al_1+\al_2+\al_3}}
\cdot
 \frac{f(p^{\be_1},p^{\be_2},p^{\be_3})p^{\be_1\delta_1}}{p^{\be_1+\be_2+\be_3}}.
$$
We may conclude that the contribution to the above 
sum from $\bal, \bbe$ such that $\bbe=\ma{0}$
is $O(p^{-1+\delta_1})$. 

Suppose now that  $\bbe\neq \ma{0}$,
with $\be_i\leq \be_j\leq \be_k$ for some permutation $\{i,j,k\}$ of
$\{1,2,3\}$ such that $\be_k\geq 1$. 
Then Lemma \ref{lem:rho} implies
that
$$
 \frac{f(p^{\be_1},p^{\be_2},p^{\be_3})p^{\be_1\delta_1}}{p^{\be_1+\be_2+\be_3}}
\leq p^{\be_1\delta_1}\cdot \frac{
p^{\min\{\be_j,v_p(\D)\}+\min\{\be_k,\la_k\}}}{p^{\be_j+\be_k}},
$$
where we have written $\la_k=v_p(\ell_k)$ for short. 
Summing this contribution over 
$\bbe\neq \ma{0}$ we therefore arrive at the contribution
\begin{align*}
&\leq 
\sum_{1\leq k\leq 3}
\sum_{\max\{\be_1,\be_2,\be_3\}=\be_k\geq 1}
p^{\be_1\delta_1}\cdot p^{\min\{\be_k,\la_k\}-\be_k}\\
&\ll
\sum_{1\leq k\leq 3}
\sum_{\be_k\geq 1}
p^{\be_k(\delta_1-1)+\min\{\be_k,\la_k\}}\\
&\ll
p^{\max\{\la_1,\la_2,\la_3\}\delta_1}.
\end{align*}
It now follows that 
$$
\prod_{p\mid \D}P_p\leq \prod_{p\mid \D} \Big(
1+O(p^{-1+\delta_1})+O(p^{\max\{\la_1,\la_2,\la_3\}\delta_1})
\Big) \ll_{\ve} L_\infty^\ve L_*^{\delta_1}, 
$$
where $L_*$ is given by \eqref{eq:Lg}. This is satisfactory for the
lemma. 

Turning to the contribution from $p\nmid \D$, it is a simple matter to
conclude that 
$$
\rho(p^{\nu_1},p^{\nu_2},p^{\nu_3};L_1,L_2,L_3)=
\rho(p^{\nu_1},p^{\nu_2},p^{\nu_3};L_1^*,L_2^*,L_3^*).
$$
Hence Lemma \ref{lem:rho} yields
$
h(p^\nu,1,1)=h(1,p^\nu,1)=h(1,1,p^\nu)=0
$
if $\nu\geq 1$ and $p\nmid \D$, since then 
$f(p^\nu,1,1)=f(1,p^\nu,1)=f(1,1,p^\nu)=1$.
Moreover, we deduce from Lemma \ref{lem:rho} and \eqref{eq:weston} that
for $p\nmid \D$ we have
\begin{align*}
|h(p^{\nu_1},p^{\nu_2},p^{\nu_3})|
&\leq
 (1+\nu_1)(1+\nu_2)(1+\nu_3)\sum_{0\leq n_i\leq
   \nu_i}f(p^{n_1},p^{n_2},p^{n_3})\\
&\leq
 (1+\nu_1)^2(1+\nu_2)^2(1+\nu_3)^2\max_{0\leq n_i\leq \nu_i}f(p^{n_1},p^{n_2},p^{n_3})
\\
&=
 (1+\nu_1)^2(1+\nu_2)^2(1+\nu_3)^2p^{\min\{\nu_1,\nu_2,\nu_3\}}.
\end{align*}
Thus
$$
\prod_{p\nmid \D}P_p=
\prod_{p\nmid \D}\Big(1+\sum_{\bnu}
\frac{(1+\nu_1)^2(1+\nu_2)^2(1+\nu_3)^2p^{\min\{\nu_1,\nu_2,\nu_3\}}}{
p^{\nu_1(1-\delta_1)+\nu_2+\nu_3}}\Big),
$$
where the sum over $\bnu$ is over all $\bnu\in \ZZ_{\geq 0}^3$ such
that  $\nu_1+\nu_2+\nu_3\geq 2$, with at least two of the variables
being non-zero. The overall contribution to the sum arising from
precisely two variables being non-zero is clearly 
$
O( p^{-2}).
$
Likewise, we see that the contribution from all three variables being
non-zero is 
$
O( p^{-2+\delta_1}).
$
It therefore follows that
$$
\prod_{p\nmid \D}P_p=
\prod_{p\nmid \D}\Big(1+O(p^{-2+\delta_1})\Big) \ll_{\delta_1,\ve} L_\infty^\ve,
$$
since $\delta_1<1$. This completes the proof of the lemma. 
\end{proof}

We are now ready to complete the proof of Lemma \ref{lem:Masy}. On
recalling the definition \eqref{eq:MM}, we see that
\begin{align*}
M(\ma{T})=\sum_{d_i\leq T_i} 
  \frac{f(\bd) }{d_1d_2d_3}=
\sum_{d_i\leq T_i} 
  \frac{(1*h)(\bd) }{d_1d_2d_3}=
\sum_{k_i\leq T_i} 
  \frac{h(\ma{k}) }{k_1k_2k_3}\sum_{e_i\leq \frac{T_i}{k_i}}\frac{1}{e_1e_2e_3}.
\end{align*}
Now the inner sum is estimated as 
$$
\prod_{i=1}^3\Big(
\log T_i- \log k_i +\gamma
+O(k_i^{\frac{1}{2}}T_i^{-\frac{1}{2}}) 
\Big).
$$
The main term in this estimate is 
equal to 
$$
\prod_{i=1}^3\log T_i +R(\log T_1,\log T_2,\log T_3),
$$
for a  quadratic polynomial $R \in \RR[x,y,z]$ with
coefficients bounded by $\ll_\ve (k_1k_2k_3)^\ve$ and no non-zero
coefficients of $x^2,y^2$ or $z^2$.
The error term is 
$\ll_\ve T^\ve \max \{k_iT_i^{-1}\}^{\frac{1}{2}}$, with $T=T_1T_2T_3.$
We may therefore apply Lemma
\ref{lem:tail} to obtain  
an overall error  of 
\begin{equation}
  \label{eq:tata1}
\ll_{\ve} L_\infty^\ve
L_*^{\frac{1}{2}}\min \{T_i\}^{-\frac{1}{2}} T^\ve,
\end{equation}
where $L_*$ is given by \eqref{eq:Lg}.

Our next step is to show that the
sums involving $\ma{k}$ can be extended to infinity with negligible
error. If $a\ll_\ve (k_1k_2k_3)^\ve$ is any of the coefficients in our 
cubic polynomial main term, then for $j\in \{1,2,3\}$ Rankin's trick yields
\begin{align*}
\sum_{\substack{\ma{k}\in \NN^3\\ k_j> T_j}} 
  \frac{|h(\ma{k})||a| }{k_1k_2k_3}
&\ll_\ve 
\sum_{\substack{\ma{k}\in \NN^3\\ k_j> T_j}} 
  \frac{|h(\ma{k})| }{(k_1k_2k_3)^{1-\ve}}
<
\frac{1}{T_j^{\frac{1}{2}}}
\sum_{\ma{k}\in\NN^3} 
  \frac{|h(\ma{k})| k_j^{\frac{1}{2}}}{(k_1k_2k_3)^{1-\ve}},
\end{align*}
which Lemma \ref{lem:tail} reveals is bounded by \eqref{eq:tata1}.
We have therefore arrived at the asymptotic formula for $M(\ma{T})$
in Lemma \ref{lem:Masy}, with 
coefficients of size
$O_\ve(L_\infty^\ve)$, as follows from Lemma~\ref{lem:tail}.
Moreover, the leading coefficient takes the shape
$$
\sum_{\ma{k}\in\NN^3} 
  \frac{h(\ma{k}) }{k_1k_2k_3}
=\sum_{\ma{k}\in\NN^3} 
  \frac{(\mu*f)(\ma{k}) }{k_1k_2k_3}
=\prod_p\sigma_p,
$$
in the notation of     \eqref{defsigma}. This therefore concludes the
proof of Lemma \ref{lem:Masy}.

\section{Theorem \ref{main1}: general case}\label{s:3}

Let  $\ma{d},\ma{D}\in \NN^3$, with $d_i \mid D_i$, and 
assume that $r' X^{1-\theta}\geq 1$
for $\theta \in (\frac{1}{4},1)$. 
In estimating $S(X;\ma{d},\ma{D})$, our goal is to replace the summation
over $\sfl(\ma{D})$  by a summation over $\ZZ^2$, in order to relate it
to the sum $S(X)$ that we studied in the previous section. 
We begin  by recording the upper bound
\begin{equation}
  \label{eq:tata2}
  S(X;\ma{d},\ma{D})\ll_\ve 
L_\infty^{\ve}r_\infty^\ve \Big(
\frac{\vol(\mcal{R})X^{2+\ve}}{\det \sfl(\ma{D})}+r_\infty X^{1+\ve}\Big).
\end{equation}
This follows immediately on taking the trivial estimate for the divisor
function and applying standard lattice point  counting results.

Given any basis $\ma{e}_1, \ma{e}_2$ for $\sfl({\ma{D}})$, let 
$M_i(\v)$ be the linear form obtained from
$d_i^{-1}L_i(\x)$ via the change of variables $\x\mapsto v_1\ma{e}_1+v_2\ma{e}_2$.
By choosing $\ma{e}_1,\ma{e}_2$ to be a minimal basis, we may  further assume that 
\begin{equation}
  \label{eq:3011.2}
1\leq |\ma{e_1}|\leq |\ma{e}_2|, \quad |\ma{e}_1||\ma{e}_2| \ll \det \sfl(\ma{D}),
\end{equation}
where $|\ma{z}|=\max|z_i|$ for $\ma{z}\in \RR^2$. 
Write $\ma{M}$ for the matrix formed from $\ma{e}_1, \ma{e}_2$. 
Carrying out this change of variables, we obtain 
\begin{align*}
S(X;\ma{d},\ma{D})
&=\sum_{\substack{\v \in \ZZ^2\cap X\mcal{R}_{\ma{M}}}}
\tau(M_1(\v))\tau(M_2(\v))\tau(M_3(\v)),
\end{align*}
where  $\mcal{R}_{\ma{M}}=\{\ma{M}^{-1}\z: \z\in \mcal{R}\}$.
Note that $M_i(\v)>0$ for every $\v$ in the summation. Moreover,
the $M_i$ will be linearly independent linear forms defined over $\ZZ$
and $\partial(\mcal{R}_{\ma{M}})\leq r_\infty(\mcal{R}_{\ma{M}})$ in
the notation of \eqref{eq:rinf}, where 
$\partial(\mcal{R}_{\ma{M}})$ is the length of the boundary of 
$\mcal{R}_{\ma{M}}.$

We now wish to estimate this quantity. 
 In view of \eqref{eq:3011.2} and the fact that $\det \sfl(\ma{D})=[\ZZ^2:\sfl(\ma{D})]$
 divides $D=D_1D_2D_3$, 
it is clear that
$$
L_\infty(M_1,M_2,M_3)\leq D L_\infty(L_1,L_2,L_3)=DL_\infty,
$$
in the notation of \eqref{eq:Linf}. 
In a similar fashion, recalling the definitions
\eqref{eq:rinf} and \eqref{eq:ri'},  we observe that 
$$
r_\infty(\mcal{R}_{\ma{M}})\ll 
\frac{|\ma{e}_1||\ma{e}_2|}{|\det \ma{M}|} r_\infty(\mcal{R}) \ll
r_\infty(\mcal{R})=r_\infty
$$
and $r'(M_1,M_2,M_3,\mcal{R}_{\ma{M}})
\leq \min \{d_1,d_2,d_3\}^{-1}r'(L_1,L_2,L_3,\mcal{R})\leq r'$.

Note that $r_\infty X\leq r_\infty
r'^{\frac{3}{4}}X^{\frac{7}{4}}$, 
by our hypothesis on $r'$. 
Moreover, since
$$
\det \sfl(\ma{D})=\frac{D^2}{\rho(\ma{D})},
$$
it follows from Lemma \ref{lem:rho} that 
$\det \sfl(\ma{D}) \gg d_k \gcd(d_k,\ell_k)^{-1}$ for any $1\leq k
\leq 3.$ Suppose for the moment that $d_k=\max\{d_i\}> X^{\frac{1}{4}}$.  Then an
application of \eqref{prelim} and \eqref{eq:tata2} easily reveals that 
\begin{equation}\label{eq:sea}
\begin{split}
S(X;\ma{d},\ma{D})
&\ll_\ve L_\infty^{\ve}r_\infty^\ve
 \Big(
\frac{r_\infty^2 X^{2+\ve}\gcd(d_k,\ell_k)}{d_k}+r_\infty
X^{1+\ve}\Big)\\
&\ll_\ve L_\infty^{\ve} r_\infty^\ve 
(r_\infty r'^{\frac{3}{4}}+
L_\infty^{\frac{1}{2}}L_*^{\frac{1}{2}} r_\infty^2) X^{\frac{7}{4}+\ve},
\end{split}
\end{equation}
where $\ell_k$ is defined in \eqref{ecritprim} and $L_*$ by
\eqref{eq:Lg}. 
Alternatively, if $\max\{d_i\}\leq X^{\frac{1}{4}}$ then for any 
$\psi>0$ we have 
$$
r'(M_1,M_2,M_3,\mcal{R}_{\ma{M}})X^{1-\psi}\geq r'
X^{\frac{3}{4}-\psi} \geq r' X^{1-\theta} \geq 1,
$$
provided that $\psi\leq \theta-\frac{1}{4}$. Taking 
$\psi=\theta-\frac{1}{4}\in (0,\frac{3}{4})$ all the hypotheses are
therefore met for an application of \eqref{eq:24}.

To  facilitate this application we  note that 
$\vol(\mcal{R}_{\ma{M}})=|\det
\ma{M}|^{-1}\vol(\mcal{R})$. 
Moreover, if $m_i$ denotes the greatest common divisor of the
coefficients of $M_i$ then 
$
m_i \mid \ell_i \det \sfl(\ma{D}).
$
Hence we have 
$$
L_*(M_1,M_2,M_3)=[m_1,m_2,m_3]\leq [\ell_1,\ell_2,\ell_3] \det
\sfl(\ma{D}) =L_* \det \sfl(\ma{D}),
$$
from which it is clear that 
$$
L_*(M_1,M_2,M_3)^{\frac{1}{2}}\vol(\mcal{R}_{\ma{M}})^{\frac{1}{2}}\leq
L_*^{\frac{1}{2}} \vol(\mcal{R})^{\frac{1}{2}}\leq 2 L_*^\frac{1}{2}r_\infty,
$$
by \eqref{prelim}. 
Finally we recall from above that
$
r'(M_1,M_2,M_3,\mcal{R}_{\ma{M}})\geq (\max\{d_i\})^{-1} r'.
$
Collecting all of this 
together, it now  follows from \eqref{eq:24} and \eqref{eq:sea} that
$$
S(X;\ma{d},\ma{D})
= \frac{\vol(\mcal{R})}{\det \sfl(\ma{D})}X^2 P(\log X) +
O_\ve\big(\mcal{E}\big),
$$
where the leading coefficient of $P$ is $\prod_{p}\sigma_p^*$ and 
$\sigma_p^*$ is defined as for 
$\sigma_p$ in \eqref{defsigma}, but with $\rho(\ma{h};L_1,L_2,L_3)$ replaced by 
$\rho(\ma{h};M_1,M_2,M_3)$, and 
$$
\mcal{E}
=
 D^\ve  L_\infty^{\ve} r_\infty^\ve 
\big(L_*^{\frac{1}{2}}r_\infty r'^{\frac{3}{4}}+ 
L_\infty^{\frac{1}{2}}L_*^{\frac{1}{2}}r_\infty^2  \big)X^{\frac{7}{4}+\ve}.
$$
Furthermore, the coefficients of $P$ are all
$O_\ve(D^\ve L_\infty^\ve r_\infty^\ve(1+r'^{-1})^\ve)$  
in modulus, so that the coefficients of the
polynomial appearing in Theorem \ref{main1} have the size claimed
there. Following the calculations in \cite[\S
6]{4linear} one finds  that 
$$
\frac{1}{\det \sfl(\ma{D})}\prod_{p}\sigma_p^* =
\prod_p \sigma_p(\ma{d},\ma{D}),
$$
in the notation of \eqref{defsigma'}.

Let us write $S(X;\ma{d},\ma{D})=S(X;\ma{d},\ma{D};L_1,L_2,L_3,\mcal{R})$ in
\eqref{eq:Sj} in order to stress the various dependencies.
Recall the notation $\delta=\delta(\ma{D})$ that was introduced prior to the
statement of Theorem \ref{main1}. 
In order to obtain the factor $\delta^{-1}$ in the error term
$\mcal{E}$ we simply observe that 
$$
S(X;\ma{d},\ma{D};L_1,L_2,L_3,\mcal{R})=
S(X;\ma{d},\ma{D};\delta L_1,\delta L_2,\delta
L_3,\delta^{-1}\mcal{R}).
$$
According to \eqref{eq:rinf} and \eqref{eq:ri'},
we see that the value of $r'$ is
left unchanged and  $r_\infty$
should be divided by $\delta$.  However, $L_\infty$ is replaced by
$\delta L_\infty$ and $L_*$ becomes $\delta L_*$. 
On noting that 
$L_*\leq \ell_1\ell_2\ell_3\leq L_\infty^3$, we easily conclude that 
the new error term is
as in Theorem \ref{main1}.
Finally the constants obtained as factors of $X^2(\log
X)^{i}$ in the main term must be the same since they are 
independent of $X$. This therefore concludes the proof of 
Theorem~\ref{main1}.

\section{Treatment of $T(X)$}\label{s:4}

In this section we establish Theorem \ref{t:d}.
For convenience we will assume that the coefficients of $L_1,L_2,L_3$
are all positive so that $L_i(\x)>0$ for all $\x\in [0,1]^2$. The
general case is readily handled by breaking the sum over $\x$ into
regions on which the sign of each $L_i(\x)$ is fixed.
In order to transfigure $T(X)$ into the sort of sum defined in
\eqref{eq:Sj}, we will follow the opening steps of the argument in
\cite[\S 7]{L1L2Q}. This hinges upon the formula 
$$
\tau(n_1n_2n_3)=\sum_{\substack{\ma{e}\in\NN^3\\ e_ie_j\mid
n_k}}\frac{\mu(e_1e_2)\mu(e_3)}{2^{\omega(\gcd(e_1,n_1))+\omega(\gcd(e_2,n_2))}} 
\tau\Big(\frac{n_1}{e_2e_3}\Big)\tau\Big(\frac{n_2}{e_1e_3}\Big)\tau\Big(\frac{n_3}{e_1e_2}\Big),  
$$
which is established in  \cite[Lemma 10]{L1L2Q} and is valid for any
$\ma{n}\in\NN^3$.  In this way we 
deduce that 
\begin{align*}
T(X)=\sum_{\ma{e}\in\NN^3} 
\mu(e_1e_2)\mu(e_3)\sum_{\substack{\ma{k}=(k_1,k_2,k_1',k_2')\in\NN^4\\
    k_ik_i'\mid e_i}}
\frac{\mu(k_1')\mu(k_2')}{2^{\omega(k_1)+\omega(k_2)}} T_{\ma{e},\ma{k} }(X),
\end{align*}
with
\begin{equation*}
  T_{\ma{e},\ma{k} }(X)=\sum_{ \x\in\mathsf{\Lambda}\cap
[0,X]^2}
\tau\Big(\frac{L_1(\x)}{e_2e_3}\Big)
\tau\Big(\frac{L_2(\x)}{e_1e_3}\Big)\tau\Big(\frac{L_3(\x)}{e_1e_2}\Big)
\end{equation*}
and $\sfl=\mathsf{\Lambda}([e_2e_3,k_1k_1'],[e_1e_3,k_2k_2'],e_1e_2)$ 
given by \eqref{eq:farm}.
Under the conditions
$k_ik_i'\mid e_i$ and $|\mu(e_1e_2)|=|\mu(e_3)|=1$,  we clearly have 
$
\mathsf{\Lambda}=\mathsf{\Lambda}{([e_2e_3,k],[e_1e_3,k],e_1e_2) },
$ 
with $k=k_1k_1'k_2k_2'$. Thus $T_{\ma{e},\ma{k} }(X)$ depends only on
$k\mid e_1e_2$.  Noting that $T_{\ma{e},\ma{k} }(X)=0$ unless
$|\ma{e}|\leq X$, 
and 
$$
\sum_{\substack{\ma{k}=(k_1,k_2,k_1',k_2')\in\NN^4\\
    k_ik_i'=\gcd(k,e_i)}}
\frac{\mu(k_1')\mu(k_2')}{2^{\omega(k_1)+\omega(k_2)}} 
=
\frac{\mu(\gcd(k,e_1))\mu(\gcd(k,e_2))}{2^{\omega(\gcd(k,e_1))+\omega(\gcd(k,e_2))}}
=\frac{\mu(k)}{2^{\omega(k)}},
$$
we may therefore write
\begin{equation}\label{T(X)=sume}
T(X)=\sum_{|\ma{e}|\leq X} \mu(e_1e_2)\mu(e_3)\sum_{ k\mid e_1e_2}
\frac{\mu(k)}{2^{\omega(k)}} T_{\ma{e},k }(X),
\end{equation}
with $T_{\ma{e},k }(X)=S(X,\ma{d},\ma{D})$ in the notation of \eqref{eq:Sj}
and
$$
\ma{d}=(e_2e_3,e_1e_3,e_1e_2),\quad
\ma{D}=([e_2e_3,k],[e_1e_3,k],e_1e_2).
$$

For the rest of this section 
we will allow all of our implied constants to depend upon $\ve$ and 
$L_1,L_2,L_3$.  In particular we may clearly 
assume that  $r_\infty=1$, $L_\infty\ll 1$ and $1\leq r'\ll 1.$
Now let $\delta=\delta(\ma{D})$ be the quantity defined in the buildup
to Theorem \ref{main1}. 
A little thought reveals that 
$$
\delta\geq [e_1',e_2',e_3',k'] \gg [e_1e_2,e_3],
$$
since $e_1e_2$ is square-free, where $e_i'=\frac{e_i}{\gcd(e_i,\D_{j,k})}$ and
$k'=\frac{k}{\gcd(k,\D_{1,2})}$ and we recall that $\D_{j,k}$ is the
resultant of $L_j,L_k$.

In view of the inequality $|\ma{e}|\leq X$,
we conclude from  Theorem \ref{main1}
that 
\begin{align*} 
T_{\ma{e},k }(X) =
  X^2P(\log X) +
   O\big( [e_1e_2,e_3]^{-1}X^{\frac{7}{4}+\ve}  \big),
\end{align*}
for a cubic polynomial $P$ with 
coefficients of size $\ll(e_1e_2e_3)^\ve [e_1e_2,e_3]^{-2}$,
since we have $\det \sfl(\ma{D})\geq \delta^2$.
The overall contribution from the error term, once
inserted into \eqref{T(X)=sume}, is
\begin{align*}
&\ll X^{\frac{7}{4}+\ve} 
\sum_{|\ma{e}|\leq X} \frac{|\mu(e_1e_2)\mu(e_3)|}{[e_1e_2,e_3]}\\
&\leq  X^{\frac{7}{4}+\ve} 
\sum_{|\ma{e}|\leq X}
\frac{\gcd(e_1e_2,e_3)}{e_1e_2e_3}\\
&= X^{\frac{7}{4}+\ve}
\sum_{e_1,e_2\leq X} \frac{1}{e_1e_2}\sum_{h\mid e_1e_2} h \sum_{\substack{e_3\leq
    X\\ h\mid e_3}} \frac{1}{e_3}\\
&\ll X^{\frac{7}{4}+\ve}.
\end{align*}
This is clearly satisfactory from the point of view of Theorem \ref{t:d}.
Similarly we deduce that the overall error produced by
extending the summation over $\ma{e}$ to infinity is 
\begin{align*}
&\ll X^{2+\ve} 
\sum_{|\ma{e}|> X} \frac{|\mu(e_1e_2)\mu(e_3)|(e_1e_2e_3)^\ve}{[e_1e_2,e_3]^2}\\
&\ll  X^{\frac{7}{4}+\ve} 
\sum_{\ma{e}\in\NN^3}
\frac{\gcd(e_1e_2,e_3)^2|\ma{e}|^{\frac{1}{4}}}{(e_1e_2e_3)^2}\\
&\ll X^{\frac{7}{4}+\ve}.
\end{align*}
This therefore concludes the
proof of Theorem \ref{t:d}.

\section{Divisor problem on average}\label{s:5}

In this section we prove Theorem \ref{t:unary}.
We begin by writing
$$
\sum_{h\leq H}\big( T_h(X)-c_hX(\log X)^3\big) =\Sigma_1-\Sigma_2,
$$
say, where $c_h$ is given by
\eqref{eq:ch}. The  following result deals with the second term.

\begin{lem}\label{lem:S2}
Let $H\geq 1$. Then we have
$$
\Sigma_2=cXH (\log X)^3 + O\big(XH^{\frac{1}{2}} (\log X)^3\big),
$$
where 
\begin{equation}
  \label{eq:c}
c=\frac{4}{3}\prod_{p>2}\Big(1+\frac{1}{p}\Big)^{-1}\Big(
1+\frac{1}{p}+\frac{1}{p^2}\Big).
\end{equation}
\end{lem}

\begin{proof}
We have $\Sigma_2=c_1 X(\log X)^3 S(H)$, where
$c_1$ is given by taking $h=1$ in \eqref{eq:ch},
and 
$
S(H)=\sum_{h\leq H}f(h),
$
with $f$ given multiplicatively by \eqref{eq:cf}. Using the equality
$f=(f*\mu)*1$ and the trivial estimate $[x]=x+O(x^{\frac{1}{2}})$, we
see that
\begin{align*}
S(H)
&=\sum_{d=1}^\infty (f*\mu)(d)\Big[\frac{H}{d}\Big]
=H\sum_{d=1}^\infty \frac{(f*\mu)(d)}{d}
+O\Big(H^{\frac{1}{2}}\sum_{d=1}^\infty \frac{|(f*\mu)(d)|}{d^{\frac{1}{2}}}\Big),
\end{align*}
provided that the error term is convergent.

For $k\geq 1$ we 
have $(f*\mu)(p^k)=f(p^k)-f(p^{k-1})$. Hence we 
calculate
$$
(f*\mu)(p^k)=
\begin{cases}
\frac{1}{p^k}\cdot
\frac{1+3k-\frac{3k}{p}-\frac{3+3k}{p^2}+\frac{3k+2}{p^3}}{(1
+\frac{2}{p})(1-\frac{1}{p})^2},
&\mbox{if $p>2$},\\
\frac{1}{2^k}\cdot (1+\frac{15k}{11}),
&\mbox{if $p=2$},
\end{cases}
$$
for $k\geq 2$, and 
$$
(f*\mu)(p)=
\begin{cases}
\frac{1}{p}\cdot
\frac{4+\frac{5}{p}}{1+\frac{2}{p}},
&\mbox{if $p>2$},\\
\frac{13}{11},
&\mbox{if $p=2$}.
\end{cases}
$$
In particular it is clear that $|(f*\mu)(p^k)|\ll kp^{-k}$, whence 
$$
\sum_{d=1}^\infty \frac{|(f*\mu)(d)|}{d^{\frac{1}{2}}} \ll_\ve  
\sum_{d=1}^\infty d^{-\frac{3}{2}+\ve} \ll 1,
$$
for $\ve<\frac{1}{2}$. It follows that 
$
S(H)=c_1'H +O(H^{\frac{1}{2}})$,
where
\begin{align*}
c_1'
&=\prod_p \sum_{k\geq 0} \frac{(f*\mu)(p^k)}{p^k} \\
&=
\frac{64}{33}\prod_{p>2}\Big(1+\frac{2}{p}\Big)^{-1}
\Big(1-\frac{1}{p}\Big)^{-2}\Big(1+\frac{1}{p}\Big)^{-1}
\Big(1+\frac{1}{p}+\frac{1}{p^2}\Big).
\end{align*}
We conclude the proof of the lemma by noting that $c_1c_1'=c$.
\end{proof}

It would be easy to replace the exponent $\frac{1}{2}$ of $H$
by any positive number, but this would not yield an overall
improvement of Theorem \ref{t:unary}. 
We now proceed to an analysis of the sum
$$
\Sigma_1=\sum_{h\leq H}T_h(X)=\sum_{\substack{h\leq H\\ n\leq X}}
  \tau(n-h)\tau(n)\tau(n+h),
$$
in which we follow the convention that $\tau(-n)=\tau(n)$. 
This corresponds to a sum of the type considered in  \eqref{eq:Sj}, 
with  $d_i=D_i=1$ and 
$$
L_1(\x)=x_1-x_2, \quad L_2(\x)=x_1, \quad L_3(\x)=x_1+x_2.
$$
The difference is that we are now summing
over a lopsided region.

\begin{lem}\label{lem:S1}
Let $H \geq 1$ and let $\ve>0$. Then we have 
$$
\Sigma_1=cXH (\log X)^3 + O_\ve\big(XH (\log X)^2+
X^{\frac{1}{2}+\ve}H
+X^{\frac{7}{4}+\ve}\big),
$$
where $c$ is given by \eqref{eq:c}.
\end{lem}

\begin{proof}
Tracing through the proof of \eqref{eq:24} one is led to consider
$8$ sums
$$
\Sigma_1^{\pm,\pm,\pm}=
\sum_{\substack{h\leq H\\ n\leq X}}
  \tau_\pm(n-h)\tau_\pm(n)\tau_\pm(n+h),
$$
with $X'=2X$ in the construction \eqref{eq:bone} of $\tau_\pm.$
Arguing as before we examine a typical sum
$$
\Sigma_1^{+,+,-}=
 \sum_{d_1,d_2,d_3\leq \sqrt{2X}}
\#(\sfl(\ma{d})\cap
\mcal{R}_{\ma{d}}(X,H)),
$$
where 
$
\mcal{R}_{\ma{d}}(X,H)=\{\x \in (0,X]\times (0,H]: d_3\sqrt{2X}<L_3(\x)\}.
$ 
An entirely analogous version of Lemma \ref{LOD} for our lopsided region
readily leads to the conclusion that 
\begin{align*}
\Sigma_1^{+,+,-}=
\sum_{d_1,d_2,d_3\leq \sqrt{2X}} 
  \frac{\rho(\bd) \vol (\mcal{R}_{\ma{d}}(X,H))}{(d_1d_2d_3)^2}
+O_\ve\big(X^{\frac{1}{2}+\ve}H+ 
X^{\frac{7}{4}+\ve}\big).
\end{align*}
Combining Lemma \ref{lem:Masy} with partial summation, as previously,
we conclude that 
\begin{align*}
\sum_{d_1,d_2,d_3\leq \sqrt{2X}} 
  \frac{\rho(\bd) \vol (\mcal{R}_{\ma{d}}(X,H))}{(d_1d_2d_3)^2}
=~&XH (\log X)^3 \prod_p \sigma_p\\
&+ O_\ve(XH(\log X)^2+X^{\frac{7}{4}+\ve}),
\end{align*}
with $\sigma_p$ given by \eqref{defsigma}.
This gives  the statement of
the lemma with $c=\prod_p \sigma_p$.

It remains to show that $c$ matches up with \eqref{eq:c}. 
Let $m(\ma{a})=\max_{i\neq j}\{a_i+a_j\}$ for any $\ma{a}\in \ZZ_{\geq 0}^3$.
For $z\in\CC$ such that $|z|<1$ we claim that
\begin{equation}
  \label{eq:car}
  S(z)=\sum_{\nu_1,\nu_2,\nu_3\geq 0} z^{m(\bnu)}=
\frac{1+z+z^2}{(1-z)^2(1-z^2)}.
\end{equation}
But this follows easily from the observation
$$
S(z)=1
+3\sum_{\substack{\nu_1=\nu_2=0\\
\nu_3\geq 1}} z^{\nu_3}
+3\sum_{\substack{\nu_1=0\\
\nu_2,\nu_3\geq 1}} z^{\nu_2+\nu_3}
+z^2S(z).
$$
Now the linear forms that arise in our
analysis have resultants
$
\Delta_{1,2}=1, \Delta_{1,3}=2$ and $\Delta_{2,3}=1.
$
Moreover, $\ell_1=\ell_2=\ell_3=1$ in the notation of \eqref{ecritprim}.
Suppose that $p>2$ and write $z=\frac{1}{p}$.  Then it follows from Lemma \ref{lem:rho} that 
\begin{align*}
\sum_{\mnu\in\ZZ_{\geq 0}^3} \frac{\rho
(p^{\nu_1},p^{\nu_2},p^{\nu_3})}{p^{2\nu_1+2\nu_2+2\nu_3}}
=S(z)=\frac{1+z+z^2}{(1-z)^2(1-z^2)}.
\end{align*}
If $p=2$, it will be necessary to revisit the proof of Lemma
\ref{lem:rho}. To begin with it is clear that 
$\rho(2^{\nu_1},2^{\nu_2},2^{\nu_3})=2^{\nu_1+\nu_2+\nu_3+\min\{\nu_i\}}$
if $\min\{\nu_1,\nu_3\}\leq \nu_2$. If $\nu_2<\nu_i\leq \nu_j$ for
some permutation $\{i,j\}$ of $\{1,3\}$ then 
\begin{align*}
\rho(2^{\nu_1},2^{\nu_2},2^{\nu_3})
=\#\{\x\bmod{2^{\nu_1+\nu_2+\nu_3}}: 2^{\nu_i}\mid \D_{1,3}\x,
~2^{\nu_j}\mid L_j(\x)\}
=2^{\nu_1+2\nu_2+\nu_3+1}.
\end{align*}
Writing $z=\frac{1}{2}$ we obtain
\begin{align*}
\sum_{\mnu\in\ZZ_{\geq 0}^3} \frac{\rho
(2^{\nu_1},2^{\nu_2},2^{\nu_3})}{2^{2\nu_1+2\nu_2+2\nu_3}}
&=
\sum_{\substack{\mnu\in\ZZ_{\geq 0}^3\\
\min\{\nu_1,\nu_3\}\leq \nu_2
 }} z^{m(\bnu)}
+
\sum_{\substack{\mnu\in\ZZ_{\geq 0}^3\\
\min\{\nu_1,\nu_3\}> \nu_2
 }}
z^{m(\bnu)-1}\\
&=S(z)
+
\sum_{\substack{\mnu\in\ZZ_{\geq 0}^3\\
\min\{\nu_1,\nu_3\}> \nu_2
 }}
z^{\nu_1+\nu_3-1}(1-z)\\
&=\frac{1+z+z^2}{(1-z)^2(1-z^2)}+\frac{z}{(1-z)^2(1+z)}.
\end{align*}
Hence, \eqref{defsigma} becomes
$$
\sigma_p =
\begin{cases}
(1+\frac{1}{p})^{-1}(
1+\frac{1}{p}+\frac{1}{p^2}), &\mbox{if $p>2$,}\\
\frac{4}{3},&\mbox{if $p=2$,}
\end{cases}
$$
as required to complete the proof of the lemma. 
\end{proof}

Once combined, Lemmas \ref{lem:S2} and \ref{lem:S1}
yield 
$$
\Sigma_1- \Sigma_2 \ll_\ve XH^{\frac{1}{2}}(\log X)^3+
XH (\log X)^2+
X^{\frac{1}{2}+\ve}H
+X^{\frac{7}{4}+\ve}.
$$
This is $o(XH (\log X)^3)$ for $H \geq X^{\frac{3}{4}+\ve}$,
as claimed in  Theorem \ref{t:unary}.

\section{Bilinear hypersurfaces}\label{s:bil}

In this section we establish Theorem \ref{thm:bil}, for which we 
begin by studying the 
counting function
$$
N_0(X)=\#\{
(\ma{u},\ma{v})\in (\ZZ\setminus\{0\})^6: 
|\ma{v}|\leq v_0\leq X^{\frac{1}{2}}, ~ 
|\ma{u}|\leq v_0^{-1}X, ~\ma{u}.\ma{v}=0\},
$$
for large $X$, where we write $|\x|=\max\{|x_0|,|x_1|,|x_2|\}$ for any
$\x=(x_0,x_1,x_2) \in \RR^3$.  The overall contribution from vectors with
$|v_1|=v_0$ is 
\begin{align*}
&\ll \sum_{|v_2|\leq v_0\leq X^{\frac{1}{2}}}  
\#\{
\ma{u}\in\ZZ^3: ~|\ma{u}|\leq v_0^{-1}X, ~u_0v_0+u_1v_0+u_2v_2=0\}\\
&\ll \sum_{|v_2|\leq v_0\leq X^{\frac{1}{2}}}  
\frac{X^2}{v_0^3}
\ll X^2,
\end{align*}
as can be seen using the geometry of numbers. 
Similarly there is a contribution of $O(X^2)$ to $N_0(X)$ from vectors
for which $|v_2|=v_0$. 
Thus we may conclude that 
$$
N_0(X)=2^3N_1(X)+O(X^2),
$$
where $N_1(X)$ is the contribution to $N_0(X)$ from 
vectors with 
$0<v_1,v_2<v_0$ and $
u_2>0$, 
with the equation $\ma{u}.\ma{v}=0$ replaced by 
$u_0v_0+ u_1v_1=u_2v_2$.

Define the region
$$
V=\Big\{\bal\in [0,1]^6: 
\al_2,\al_3< \al_1\leq \frac{1}{2}, ~\al_1+\al_5-\al_2\leq 1, ~\al_1+\al_6-\al_3\leq 1\Big\},
$$
and set 
$$
 L_1(\x)=x_1, \quad  L_2(\x)=x_2, \quad  L_3(\x)=x_1+ x_2.
$$
We will work with the  region
$
\mcal{R}=\{\x\in [-1,1]^2: x_1x_2\neq 0, ~x_1+x_2>0\}.
$
Then we clearly have 
$
N_1(X)
=R(X),
$ 
with 
$$
R(X)=
\sum_{\substack{\x\in \ZZ^2\cap
     X\mcal{R}}}
\#\left\{
\ma{e}\in \NN^3: e_i \mid L_i(\x),~ (\bep,\bxi)\in V\right\},
$$
where $\bep=(\epsilon_1,\epsilon_2,\epsilon_3), \bxi=(\xi_1,\xi_2,\xi_3)$ and 
$$
\epsilon_i=\frac{\log e_i}{\log X}, \quad \xi_i=
\frac{\log |L_i(\x)|}{\log  X}.
$$
Note that for $V=[0,1]^6$ this sum coincides with \eqref{eq:Sj} for
$d_i=D_i=1$.
We establish an asymptotic formula for 
$R(X)$ along the lines of the proof 
of Theorem~\ref{main1}. 
We will need to 
arrange things so that we are only considering small divisors in the
summand. It is easy to see that the overall contribution to the sum
from $\mathbf{e}$ such that $e_j^2=L_j(\x)$ for some $j\in
\{1,2,3\}$ is 
\begin{align*}
 \ll_\ve X^\ve  \sum_{e_j\leq \sqrt{X}} \#\big\{
\x\in \ZZ^2\cap X\mcal{R}:  ~L_j(\x)=e_j^2\big\}
\ll_\ve X^{\frac{3}{2}+\ve}.
\end{align*}
It follows that we may write
\begin{equation}
  \label{eq:speak}
R(X)=
\sum_{\ma{m}\in \{\pm 1\}^3} R^{(\ma{m})}(X) +O_\ve(X^{\frac{3}{2}+\ve}),
\end{equation}
where $R^{(\ma{m})}(X)$ is the contribution from $
m_i e_i\leq m_i\sqrt{|L_i(\x)|}.
$

We  indicate how to get an asymptotic formula for
$R^{(1,1,-1)}(X)=R^{+,+,-}(X)$, say, which is
typical. Writing $L_3(\x)=e_3f_3$, we see that $f_3\leq
\sqrt{L_3(\x)}$ and  
$$
\epsilon_3=\frac{\log(f_3^{-1}L_3(\x))}{\log X}=\xi_3-\frac{\log f_3}{\log X}.
$$
On relabelling the variables we may therefore write
$$
R^{+,+,-}(X)=
\sum_{\substack{\x\in \ZZ^2\cap
    X\mcal{R}    }}
\#\left\{
\ma{e}\in \NN^3: 
\begin{array}{l}
e_i \mid L_i(\x), ~e_i \leq \sqrt{|L_i(\x)|}\\ 
(\bep,\bxi)\in V^{+,+,-}
\end{array}
\right\},
$$
where 
$$
V^{+,+,-}=\{
(\bep,\bxi)\in \RR^6: 
(\epsilon_1,\epsilon_2,\xi_3-\epsilon_3, \bxi)\in V
\}.
$$
Interchanging the order of summation we obtain
$$
R^{+,+,-}(X)=
\sum_{\substack{\mathbf{e}\in \NN^3}}
\#\left\{
\x\in \sfl(\ma{e})\cap
     X\mcal{R}:
\bxi\in V^{+,+,-}(\mathbf{e})
\right\},
$$
where $\bxi\in V^{+,+,-}(\mathbf{e})$ if and only if 
$(\bep,\bxi)\in V^{+,+,-}$ and 
$2\epsilon_i\leq \xi_i$.

On verifying that the underlying region is a union of two convex
regions, an application of Lemma~\ref{LOD} yields
$$
R^{+,+,-}(X)=\sum_{\substack{\mathbf{e}\in \NN^3}}
\frac{\vol\{\x \in X\mcal{R}: \bxi\in
  V^{+,+,-}(\mathbf{e})\}\rho(\mathbf{e})}
{(e_1e_2e_3)^2}
+O_\ve(X^{\frac{7}{4}+\ve}).
$$
Lemma \ref{lem:rho} implies that 
$$
\frac{\rho(\mathbf{e})}{e_1e_2e_3} =\gcd(e_1,e_2,e_3)=f(\ma{e}),
$$
say, whence
$$
R^{+,+,-}(X)
=
\int_{\x\in X\mcal{R}}
\sum_{\substack{\mathbf{e}\in \NN^3\\ 2\epsilon_i\leq \xi_i}}
\frac{\chi_V(\epsilon_1,\epsilon_2,\xi_3-\epsilon_3, \bxi)
f(\ma{e})}{e_1e_2e_3} \d\x
+O_\ve(X^{\frac{7}{4}+\ve}),
$$
where $\chi_V$ is the characteristic function of the set $V$.
We now write $f=h*1$ as a convolution, 
for a multiplicative arithmetic function $h$.
Opening it up gives
\begin{equation}\label{eq:focus}
R^{+,+,-}(X)=
\sum_{\substack{\ma{k}\in\NN^3}}
\frac{h(\ma{k})}{k_1k_2k_3} 
\int_{\x\in X\mcal{R}} 
M(X)\d \x +O_\ve(X^{\frac{7}{4}+\ve}),
\end{equation}
where for $\kappa_i=\frac{\log k_i}{\log X}$ we set
$$
M(X)=
\sum_{\substack{\mathbf{e}\in \NN^3\\ 2\epsilon_i+2\kappa_i\leq \xi_i}}
\frac{\chi_V(\epsilon_1+\kappa_1,\epsilon_2+\kappa_2,
\xi_3-\epsilon_3-\kappa_3, \bxi)}{e_1 e_2e_3}.
$$
The estimation of $M(X)$ will depend intimately 
on the set $V$. Indeed we wish to show
that $\int M(X)\d\x$ has order $X^2\log
X$, whereas taking $V=[0,1]^6$ leads to a sum with order
$X^2(\log X)^3$.

Writing out the definition of the set $V$ we see that 
$$
M(X)=
\sum_{\substack{e_1\in \NN\\
0\leq \epsilon_1+\kappa_1\leq \frac{1}{2}\\
2\epsilon_1+2\kappa_1\leq \xi_1}}
\frac{1}{e_1}
\sum_{\substack{e_2\in \NN\\
0\leq \epsilon_2+\kappa_2< \epsilon_1+\kappa_1\\
\epsilon_1+\kappa_1+ \xi_2\leq 1+\epsilon_2+\kappa_2\\
2\epsilon_2+2\kappa_2\leq \xi_2}}
\frac{1}{e_2}
\sum_{\substack{e_3\in \NN\\
\xi_3< \epsilon_1+\kappa_1+ \epsilon_3+\kappa_3\leq 1\\
2\epsilon_3+2\kappa_3\leq \xi_3}}
\frac{1}{e_3},
$$
where $\epsilon_i=\frac{\log e_i}{\log X}$,
$\kappa_i=\frac{\log k_i}{\log X}$ and $\xi_i=
\frac{\log |L_i(\x)|}{\log  X}$.
Further thought shows that the outer sum over $e_1$ can actually be
taken over $e_1$ such that 
$$
\frac{\xi_3}{2}< \epsilon_1+\kappa_1 \leq \min\Big\{\frac{1}{2},
\frac{\xi_1 }{2}, 1-\frac{\xi_2}{2}\Big\}.
$$
The inner sums over $e_2,e_3$ can be approximated simultaneously by integrals, giving
$$
\Big(\log X
\int_{\max\{0,\epsilon_1+\kappa_1+\xi_2-1\}}^{\min\{\epsilon_1+\kappa_1, \frac{\xi_2}{2}\}}
\d \tau_2 +O(1)\Big)
\Big(\log X
\int_{\max\{0,\xi_3-\epsilon_1-\kappa_1\}}^{\min\{1-\epsilon_1-\kappa_1, \frac{\xi_3}{2}\}}
\d \tau_3 +O(1)\Big),
$$
after an obvious change of variables.
We see that the overall contribution to $M(X)$ from the error terms is 
\begin{align*}
&\ll 
\log X \int_{\frac{\xi_3}{2}}^{\frac{\xi_1}{2}} \Big(
1+\log X\int_{\xi_2+\tau_1-1}^{\tau_1}\d\tau_2
+\log X\int_{\xi_3-\tau_1}^{\frac{\xi_3}{2}}\d \tau_3\Big)\d\tau_1\\
&= 
(I_1+I_2+I_3)\log X,
\end{align*}
say. Let $\mcal{I}_i$ denote the integral of $I_i \log X$ over 
$\x\in X\mcal{R}$. We see that 
\begin{align*}
\mcal{I}_1
&\leq
\frac{1}{2}
\int_{\{\x\in  X\mcal{R}: ~x_1+ x_2 <|x_1|\}} 
\big(\log |x_1| -\log(x_1+ x_2)\big)
\d\x \ll X^2. 
\end{align*}
Next we note that 
\begin{align*}
\mcal{I}_2
&\ll (\log X)^2 \int_{(\tau_1,\tau_2)\in [0,\frac{1}{2}]^2}
\int_{\{\x\in  X\mcal{R}: ~\xi_3\leq 2\tau_1, ~\xi_2\leq
  1+\tau_2-\tau_1, ~x_2>0\}}
\d\x \d\tau_1 \d\tau_2\\
&\leq (\log X)^4 \int_{(\tau_1,\tau_2)\in [0,\frac{1}{2}]^2} 
\int_{-\infty}^{2\tau_1}\int_{-\infty}^{1+\tau_2-\tau_1} X^{u+v}\d u\d v \d \tau_1 \d \tau_2\\
&= (\log X)^2 \int_{(\tau_1,\tau_2)\in [0,\frac{1}{2}]^2} 
X^{1+\tau_1+\tau_2} \d \tau_1 \d \tau_2 \ll X^2,
\end{align*}
and likewise,  
\begin{align*}
\mcal{I}_3
&\ll (\log X)^2  
\int_{(\tau_1,\tau_3)\in
  [0,\frac{1}{2}]^2}
\int_{\{\x\in  X\mcal{R}: ~
\xi_2\leq 1, ~\xi_3\leq \tau_1+\tau_3, ~x_2>0\}}
\d\x \d\tau_1 \d\tau_3\\
&\leq (\log X)^4 
\int_{(\tau_1,\tau_3)\in
  [0,\frac{1}{2}]^2}
\int_{-\infty}^{1}\int_{-\infty}^{\tau_1+\tau_3} X^{u+v}\d u\d v \d \tau_1 \d \tau_3\ll X^2.
\end{align*}
Interchanging the sum over $e_1$ with the integrals over
$\tau_2,\tau_3$ one uses the 
same sort of argument to show that the 
final summation can be approximated by an
integral.  

This therefore leads to the conclusion that 
$$
\int_{\x\in X\mcal{R}} M(X)
 \d\x =
(\log X)^3
\int_{\x\in X\mcal{R}}
\int_{\substack{2\tau_1\leq \xi_1\\
2\tau_2\leq \xi_2\\
2\tau_3> \xi_3}}
\chi_V(\bta ,\bxi)\d \bta\d\x
+O(X^2),
$$
after an obvious change of variables. 
We insert this into \eqref{eq:focus} and then, on assuming analogous
formulae for all the sums $R^{\pm,\pm,\pm}(X)$, we
sum over all of the
various permutations of $\mathbf{m}$ in 
\eqref{eq:speak}. This gives
$$
R(X)
= c_0I(X)
+O(X^2),
$$
where
$$
c_0=\sum_{\substack{\ma{k}\in\NN^3}}
\frac{h(\ma{k})}{k_1k_2k_3}, \quad
I(X)=
(\log X)^3
\int_{\x\in X\mcal{R}}
\int_{\bta\in \RR^3}
\chi_V(\bta ,\bxi)\d \bta\d\x.
$$
Recalling \eqref{eq:car} we easily deduce that 
\begin{align*}
c_0
&=
\sum_{\ma{a}\in \NN^3}\frac{\mu(a_1)\mu(a_2)\mu(a_3)}{a_1a_2a_3}
\sum_{\ma{b}\in \NN^3} \frac{\gcd(b_1,b_2,b_3)}{b_1b_2b_3}\\
&=
\prod_p \Big(1-\frac{1}{p}\Big)^3
S\Big(\frac{1}{p}\Big)\\
& 
= 
\prod_p
\Big(1+\frac{1}{p}\Big)^{-1}\Big(1+\frac{1}{p}+\frac{1}{p^2}\Big).
\end{align*}

It remains to analyse the term
\begin{align*}
I(X)&=
(\log X)^3\vol\left\{
(\x,\bta)\in \RR^2\times [0,1]^3:
\begin{array}{l}
x_1+x_2>0, ~|x_1|\leq X,\\
\tau_2,\tau_3<\tau_1\leq \frac{1}{2},\\
\frac{\log |x_2|}{\log X}\leq 1+\tau_2-\tau_1,\\
\frac{\log x_1+x_2}{\log X}\leq 1+\tau_3-\tau_1
\end{array}
\right\}\\
&= I^{+,+}(X)+I^{-,+}(X)+I^{+,-}(X),
\end{align*}
where $I^{+,+}(X)$ (resp. $I^{-,+}(X)$, $I^{+,-}(X)$)
is the contribution from $\x,\bta$ such that  $x_1>0$ and $x_2>0$
(resp. $x_1<0$ and $x_2>0$, $x_1>0$ and $x_2<0$). 
In the first integral it is clear that $x_1<x_1+x_2\leq X$
so  that the condition $|x_1|\leq X$ is implied by the
others. Likewise, in the second volume integral we will have 
$x_2>|x_1|$ and so the condition $|x_1|\leq X$ is implied by the
inequalities involving $x_2$. 
An obvious change of variables readily leads to the conclusion that 
$I^{+,+}(X)+I^{-,+}(X)$ is
\begin{align*}
&=
(\log X)^5\int_{\{\bta \in [0,\frac{1}{2}]^3: \tau_2,\tau_3<\tau_1\}}
\int_{-\infty}^{1+\tau_3-\tau_1}
 \int_{-\infty}^{1+\tau_2-\tau_1} 
\hspace{-0.5cm}X^{u+v}\d u\d v \d\bta\\
&=
X^2(\log X)^3\int_{\{\bta \in [0,\frac{1}{2}]^3: \tau_2,\tau_3<\tau_1\}}
X^{\tau_2+\tau_3-2\tau_1}\d\bta\\
&= \frac{1}{2}X^2\log X +O(X^2).
\end{align*}
The final integral $I^{+,-}(X)$ can 
be written as in the first line of the above, but 
with the added constraint that
$X^u+X^v\leq X$ in the inner integration over $u,v$. For 
large $X$ this constraint can be dropped with acceptable
error, which thereby leads to the companion estimate
$$
I^{+,-}(X)
=
\frac{1}{2}X^2\log X +O(X^2).
$$

Putting everything together we have therefore shown that 
$$
N_0(X)=2^3N_1(X)+O(X^2)=8c_0 X^2\log X+O(X^2),
$$
with $c_0$ given above. 
Running through the reduction steps in \cite[\S 5]{spencer} rapidly
leads from this asymptotic formula to the statement of 
Theorem \ref{thm:bil}.

\subsection*{Acknowledgments}
It is pleasure to thank the referee for carefully reading the
manuscript and making numerous helpful suggestions. The author is
indebted to both the referee and Daniel Loughran for pointing out an
oversight in the earlier treatment of Theorem~\ref{main1}.
This 
work is supported by the NSF under agreement
\texttt{DMS-0635607} and
EPSRC grant number
\texttt{EP/E053262/1}. It 
was undertaken while the author was visiting the 
{\em Hausdorff Institute} in Bonn and the 
{\em Institute
for Advanced Study} in Princeton, the hospitality and financial
support of which are gratefully acknowledged.

\end{document}